\documentclass[12pt,twopage]{paper}

\usepackage{times}
\usepackage{amssymb}
\usepackage{amsfonts}
\usepackage{amsmath}
\usepackage{bold-extra}
\usepackage{pifont}
\usepackage{yfonts}
\usepackage{url}
\usepackage{hyperref}
\usepackage{fancyhdr}

\newtheorem{theorem}{\textbf{\textsc{Theorem}}}[section]
\newtheorem{definition}[theorem]{\textbf{\textsc{Definition}}}

\newtheorem{remark}[theorem]{\textbf{\textsc{Remark}}}

\newenvironment{proof}
{\noindent\mbox{\textsf{\textbf{\textsc{Proof}}}:}}
{\hfill{\scriptsize \mbox{\underline{\tt\textbf{QED}\,}$\!|$}}\bigskip}

\title{Tarski's Undefinability Theorem and Diagonal Lemma}

\author{{\sc Saeed \  Salehi  } \\
\textrm{\large Research Institute for Fundamental Sciences \textup{(RIFS)}, University of Tabriz,}\\ \textrm{\large P.O.Box~51666--16471, Tabriz, IRAN.  \, E-mail:\!~\textsf{\normalsize salehipour@tabrizu.ac.ir}} }

\begin{document}

\maketitle

\pagestyle{fancy}
\lhead[ ]{ \thepage\quad{\sc Saeed Salehi} ({\sf 2020}) }
\chead[ ]{ }
\rhead[ ]{ {\em Tarski's Undefinability Theorem and Diagonal Lemma} }
\lfoot[ ]{ }
\cfoot[ ]{ }
\rfoot[ ]{ }

\begin{abstract}
We prove the equivalence of the semantic version of Tarski's theorem on the undefinability of truth with a semantic version  of the Diagonal Lemma, and also show the equivalence of syntactic Tarski's Undefinability Theorem with a weak syntactic diagonal lemma. We outline two seemingly diagonal-free proofs for these theorems from the literature, and show that syntactic Tarski's theorem can deliver  G\"odel-Rosser's Incompleteness Theorem.


\end{abstract}
%



\section{Introduction}\label{sec:intro}

{\Large O}{\large\textsc{ne}}\textsc{ of the cornerstones}  of modern logic (and theory of incompleteness after G\"odel) is the Diagonal Lemma (aka Self-Reference, or Fixed-Point Lemma) due to G\"odel and Carnap (see \cite{salehi} and the references therein). The lemma states that (when  $\alpha\mapsto\ulcorner\alpha\urcorner$ is a suitable G\"odel coding  which assigns the closed term $\ulcorner\alpha\urcorner$ to a syntactic expression or object $\alpha$) for a given formula $\Psi(x)$ with the only free variable $x$, there exists some sentence $\theta$ such that the equivalence $\Psi(\ulcorner\theta\urcorner)\leftrightarrow\theta$ holds; ``holding'' could mean either {\sl being true in the standard model of natural numbers $\mathbb{N}$} or {\sl being provable in a suitable theory $T$} (which is usually taken to be a consistent extension of Robinson's arithmetic). When the equivalence $\Psi(\ulcorner\theta\urcorner)\leftrightarrow\theta$ holds in $\mathbb{N}$   we call it {\em the Semantic Diagonal Lemma} (studied in Section~\ref{sec:sema}); when $T$ proves the equivalence, we call it {\em the Syntactic Diagonal Lemma}. The {\em Weak Diagonal Lemma}  states the consistency of the sentence
$\Psi(\ulcorner\theta\urcorner)\leftrightarrow\theta$ with   $T$, for some sentence $\theta$ which depends on the given arbitrary formula $\Psi(x)$ and the theory $T$
 (studied in Section~\ref{sec:synt}).

The  Diagonal Lemma has been used in proving  many fundamental    theorems of mathematical logic, such as G\"odel's First and (also) Second Incompleteness Theorems, Rosser's (strengthening of G\"odel's Incompleteness) Theorem, and Tarski's Theorem (on the Undefinability of Truth). One problem with the Diagonal Lemma is its standard proof which is a kind of magic (or ``pulling a rabbit out of the hat'', see e.g. \cite{wass}); indeed it is not easy to remember its typical proof, even after several years of teaching it. Here, we quote some texts on the proof of this lemma from the literature:
\begin{itemize}\itemindent=1em
\item[(1998)]\!S.\ Buss writes in \cite{buss} that the proof of the Diagonal Lemma is
``quite simple but rather tricky and difficult to
conceptualize.''
\medskip
\item[(2002)]\!V.\ McGee  states in \cite{mcgee}  that by  the diagonal (aka self-referential) lemma   there  exists  a sentence $\phi$ for a given formula $\Psi(x)$ such that $\Psi(\ulcorner\phi\urcorner)\leftrightarrow\phi$ is provable in Robinson's arithmetic.
``You would hope that such a deep theorem would have an insightful proof. No such luck. I am going to write down a sentence
$\phi$
 and verify that it works.  What I won't do is give you a
satisfactory explanation for why I write down the particular formula I do. I write down the
formula because G\"odel wrote down the formula, and G\"odel wrote down the formula because,
when he played the logic game he was able to see seven or eight moves ahead, whereas you and
I are only able to see one or two moves ahead. I don't know anyone who thinks he has a fully
satisfying understanding of why the Self-referential Lemma works. It has a rabbit-out-of-a-hat
quality for everyone."
\medskip
\item[(2004)]\!H.\ Kotlarski  \cite{kotlarski04} said  that
the diagonal
lemma ``being very intuitive in the natural language,
is highly unintuitive in formal theories like Peano arithmetic. In fact,
the usual proof of the diagonal lemma [...] is short, but tricky and
difficult to conceptualize. The problem was to eliminate this lemma
from proofs of G\"odel's result.''
\medskip
\item[(2006)]\!G.\  Ser\'eny  \cite{sereny06} attempts to make ``the proof of the lemma completely transparent by showing that it is simply a straightforward
translation of the Grelling paradox into first-order arithmetic.''
\medskip
\item[(2006)]   H.\ Gaifman mentions in \cite{gaifman} that the proof of the Diagonal Lemma  is ``extremely short''. However, the ``brevity of the proof does not make for transparency; it has the aura of a magician's trick.''
\end{itemize}

In this paper, we attempt at giving some explanations and motivations for this basic lemma, in a way that   we will have  a satisfactory understanding for at least some weaker versions of it. For that purpose, we will first  see the equivalence of the semantic form of the diagonal lemma with Tarski's theorem on the undefinability of arithmetical truth (in Section~\ref{sec:sema}).
In other words, the diagonal lemma holds in the standard model of natural numbers just because the set of (the G\"odel codes of) the true arithmetical sentences is not definable.
As a matter of fact, different proofs for Tarski's undefinability theorem can lead to different proofs for the semantic version of this lemma. We will review two such proofs (presented in \cite{caicedo,kossak,kotlarski98,kotlarski04,sereny04}) which are supposedly diagonal-free. Having different proofs will, hopefully, shed some new light on the nature of this lemma and will increase our understanding about it. Then, secondly, we will see that a syntactic version of Tarski's theorem is equivalent to a weak (syntactic) version of the diagonal lemma (in Section~\ref{sec:synt}). This weak form of the diagonal lemma is still sufficiently strong to prove G\"odel-Rosser's
incompleteness theorem. So, different proofs of the syntactic version of Tarski's theorem   will provide some   seemingly diagonal-free proofs   for  Rosser's theorem  (cf.~\cite{visser19}, in which G\"odel's second incompleteness theorem is derived from Tarski's undefinability
theorem by some circular-free arguments).

\section{The Diagonal Lemma, Semantically}
\label{sec:sema}

\begin{definition}[Semantic Diagonal Lemma]\label{def:semdl}
\noindent

\noindent
The following statement is called the {\em Semantic Diagonal Lemma}:

\qquad
{\sl For every formula $\Psi(x)$ there exists a sentence $\theta$ such that $\mathbb{N}\vDash\Psi(\ulcorner\theta\urcorner)
\leftrightarrow\theta$.}
\hfill\ding{71}\end{definition}

This (weaker) form of the Diagonal Lemma serves to prove  the  semantic version  of G\"odel's Incompleteness Theorem (see~\cite[Theorem~6.3]{smith}):

\begin{theorem}[G\"odel's Incompleteness Theorem for Sound  Definable Theories]\label{thm:g1}
\noindent

\noindent
For every definable and sound theory $T$ there exists a  true sentence  independent from $T$.
\end{theorem}
\begin{proof}

\noindent
If $T$ is definable then there exists a formula ${\sf Pr}_T(x)$ such that for every sentence $\eta$ we have $T\vdash\eta$ if and only if $\mathbb{N}\vDash{\sf Pr}_T(\ulcorner\eta\urcorner)$. Now, by the Semantic Diagonal Lemma
we have $\mathbb{N}\vDash\gamma\leftrightarrow\neg{\sf  Pr}_T(\ulcorner\gamma\urcorner)$ for some
sentence $\gamma$. It can be seen that $T\nvdash\gamma$, since $T\vdash\gamma$ implies on the one hand that $\mathbb{N}\vDash{\sf Pr}_T(\ulcorner\gamma\urcorner)$, and on the other hand (by the soundness of $T$) that $\mathbb{N}\vDash\gamma$ and so $\mathbb{N}\vDash\neg{\sf Pr}_T(\ulcorner\gamma\urcorner)$, a contradiction. So, $T\nvdash\gamma$, therefore $\mathbb{N}\vDash\neg{\sf Pr}_T(\ulcorner\gamma\urcorner)$, whence $\mathbb{N}\vDash\gamma$, which also implies (by the soundness of $T$) that $T\nvdash\neg\gamma$.
\end{proof}

Also, Tarski's Theorem on the Undefinability of (arithmetical) Truth follows from the Semantic Diagonal Lemma (see~\cite[Exercise~3.7]{kaye} and cf.~\cite[Chapter~9]{kaye}):

\begin{theorem}[Tarski's   Theorem on the Undefinability of Arithmetical Truth]\label{thm:semt}
\noindent

\noindent
The G\"odel codes of the set of true sentences, i.e.  $\{\ulcorner\eta\urcorner
\mid\mathbb{N}\vDash\eta\}$,  is not definable in $\mathbb{N}$.
\end{theorem}
\begin{proof}

\noindent
If
$\{\ulcorner\eta\urcorner
\mid\mathbb{N}\vDash\eta\}$ is definable by
some
$\Gamma(x)$, then
$\mathbb{N}\vDash\Gamma(\ulcorner\theta\urcorner)
\!\leftrightarrow\!\theta$ holds
for every sentence $\theta$. Now, by the Semantic Diagonal Lemma
$\mathbb{N}\vDash\neg\Gamma
(\ulcorner\lambda\urcorner)
\!\leftrightarrow\!\lambda$ holds for a sentence
$\lambda$; so we have $\mathbb{N}\vDash\lambda\!\leftrightarrow\!
\Gamma(\ulcorner\lambda\urcorner)
\!\leftrightarrow\!\neg\lambda$, which is a contradiction.
\end{proof}

The fact of the matter is that  the Semantic Diagonal Lemma is equivalent to Tarski's Theorem on the Undefinability of Truth and to Semantic Incompleteness Theorem of G\"odel:

\begin{theorem}[Semantic Diagonal Lemma
$\!\!\!\boldsymbol\iff\!\!\!$ 
Semantic G\"odel's  Theorem
$\!\!\!\boldsymbol\iff\!\!\!$ Tarski's  Theorem]\label{thm:semequ}

\noindent
The following statements are equivalent:

1. Semantic Diagonal Lemma (Definition~\ref{def:semdl});

2. Semantic G\"odel's Incompleteness Theorem (\ref{thm:g1});

3. Semantic Tarski's Undefinability Theorem (\ref{thm:semt}). 
\end{theorem}
\begin{proof}

\noindent
The implication ($1\Rightarrow 2$) is proved in Theorem~\ref{thm:g1} (and $1\!\Rightarrow\!3$ is proved in Theorem~\ref{thm:semt}). 

($2\Rightarrow 3$): If ${\rm Th}(\mathbb{N})=\{\eta\mid\mathbb{N}\vDash\eta\}$ were definable, then since it is sound, there would be some sentence independent from it (by 2); but it is a complete theory.

($3\Rightarrow 1$):
Suppose that the set $\{\ulcorner\theta\urcorner\mid\mathbb{N}
\vDash\theta\}$ is not definable by any formula. Then for a given formula $\Psi(x)$ the formula $\neg\Psi(x)$ cannot define this set, and so we cannot have  $\mathbb{N}\vDash
\neg\Psi(\ulcorner\beta\urcorner)
\!\iff\!\mathbb{N}\vDash\beta$, for all sentences $\beta$; whence there should exists some sentence $\theta$ such that $\mathbb{N}\nvDash\neg\Psi(\ulcorner\theta\urcorner)
\leftrightarrow\theta$. Now, by the classical propositional   tautology $\neg(p \leftrightarrow q)\equiv (\neg p \leftrightarrow q)$, we have $\mathbb{N}\vDash\Psi(\ulcorner\theta\urcorner)
\leftrightarrow\theta$. So, for every $\Psi(x)$ there exists some $\theta$ for which the equivalence $\Psi(\ulcorner\theta\urcorner)
\leftrightarrow\theta$ holds in $\mathbb{N}$.
\end{proof}

So, after all, Tarski's Undefinability Theorem (in its semantic form) is not very much different from the Diagonal Lemma (in the semantic form). Therefore, it may seem at the first glance that the only way to prove Tarski's theorem is to use the Diagonal Lemma (as is done in almost all the textbooks). But as a matter of fact, there are some, supposedly, diagonal-free proofs for Tarski's theorem in the literature (see e.g. \cite{kossak}) which by Theorem \ref{thm:semequ} can give us some diagonal-free proofs for the Diagonal Lemma itself! We will outline two of them  below.

\[
\textrm{ Assume  that } \Upsilon(x)  \textrm{ defines truth in } \mathbb{N}; \textrm{  i.e., }
\mathbb{N}\vDash\Upsilon(\ulcorner\beta\urcorner)
\leftrightarrow\beta \textrm{ for all sentences } \beta.
 \tag{$\textgoth{T}$} \label{eq:special}
\]

\subsection{\bf The First Proof}\label{p1}
\paragraph{\textsc{Convention}:}
Let us make the convention that all the individual variables of our syntax are $x,x',x'',x''',\cdots$ whose lengths are $1,2,3,4,\cdots$, respectively.
 By this convention, there will be at most finitely many formulas with length $n$ for a given $n\!\in\!\mathbb{N}$ (otherwise  the formulas $x\!=\!x,\;y\!=\!y,\;z\!=\!z,\cdots$ all would have  length three).

 \begin{definition}[${\tt len}(x),\overline{n}$, {\sf definability}, ${\tt D}(x),{\tt Def}_\Upsilon^{<z}(y),{\tt Berry}_\Upsilon^{<v}(u),\ell_\Upsilon,{\tt B}_\Upsilon(x)$]\label{def:berry}
\noindent

\begin{itemize}
\item   Let ${\tt len}(x)$ denote the {\em length} of the formula with G\"odel code $x$.
\smallskip
\item  For $n\!\in\!\mathbb{N}$,   let $\overline{n}$ be the term that {\em represents} the number $n$, i.e., $\overline{0}=0$, $\overline{1}=1$, and for every $m\!\geqslant\!1$ we have   $\overline{m+1}=1+(\overline{m})$.
\smallskip
\item    We say that a number $n\!\in\!\mathbb{N}$ is {\em definable} by the  formula $\varphi(x)$, in which $x$ is the only free variable, when $\forall \zeta[\varphi(\zeta)\leftrightarrow\zeta\!=\!\overline{n}]$ is true \textup{(}in $\mathbb{N}$\textup{)}.
\smallskip
\item  Let ${\tt D}(x,y)$ be the G\"odel code of the formula which states  that {\em
the formula with G\"odel code $x$ defines
the number $y$}; so, ${\tt D}(\ulcorner\varphi\urcorner,y) =
\ulcorner\forall\zeta[\varphi(\zeta)\leftrightarrow\zeta\!=\!y]\urcorner$.
\smallskip
\item Let ${\tt Def}_\Upsilon^{<z}(y)$ be the formula  $\exists\alpha\;\big({\tt Formula}(\alpha) \;\wedge\; {\tt len}(\alpha)\!<\!z\;\wedge\;
\Upsilon[{\tt D}(\alpha,y)]\big)$
which states that  {\em the number $y$ is definable by a formula with length  less than $z$} if $\Upsilon$ is a truth predicate; needless to say, ${\tt Formula}(\alpha)$ states that $\alpha$ is the G\"odel code of a formula.
\smallskip
\item Let ${\tt Berry}_\Upsilon^{<v}(u)$ be the formula $\neg{\tt Def}_\Upsilon^{<v}(u)\;\wedge\; \forall w\!<\!u\,{\tt Def}_\Upsilon^{<v}(w)$, which states  that  {\em $u$ is the least number not defined by a formula with length less than $v$}.
\smallskip
\item Let $\ell_\Upsilon$ be the length of the formula   ${\tt Berry}_\Upsilon^{<x'}(x)$.
\smallskip
\item Let  ${\tt B}_\Upsilon(x)$ be the formula $\exists x' \big[ x'\!=\!\overline{6}\cdot\overline{\ell_\Upsilon} \;\wedge\;
{\tt Berry}_\Upsilon^{\;<x'}(x)\big]$.
\hfill\ding{71}
\end{itemize}
\end{definition}

\noindent
Here is an alternative proof (from \cite{caicedo,kotlarski04,sereny04}) for contradicting (\ref{eq:special}):

\bigskip

\begin{proof}

\noindent
The length of  ${\tt B}_\Upsilon(x)$ is  less than $6\ell_\Upsilon$; since it can be seen to be equal to $10\!+\!{\tt len}(\overline{5})\!+\!{\tt len}(\overline{\ell_\Upsilon})\!+\!\ell_\Upsilon
\!=\!24\!+\!5\ell_\Upsilon$, as we have ${\tt len}(\overline{m})\!=\!4m\!-\!3$ for every $m\!\geqslant\!1$. So, the formula ${\tt B}_\Upsilon(x)$ with length less than $6\ell_\Upsilon$ states that $x$ is the least number that is not definable by any formula with length less than $6\ell_\Upsilon$.  Whence, if ${\tt B}_\Upsilon(x)$ holds, then $x$ should not be definable by ${\tt B}_\Upsilon(x)$ itself. But this is a contradiction, since if ${\tt B}_\Upsilon(x)$ holds, then $x$ is definable by
${\tt B}_\Upsilon(\zeta)$. That is   because  ${\tt B}_\Upsilon(x)$ implies $\forall\zeta[{\tt B}_\Upsilon(\zeta)\leftrightarrow\zeta\!=\!x]$ by the sentence  $\forall u,v[{\tt B}_\Upsilon(u)\wedge{\tt B}_\Upsilon(v)\rightarrow u\!=\!v]$, which follows in turn from the sentence $\forall u,v,w[{\tt Berry}_\Upsilon^{<w}(u)\wedge{\tt Berry}_\Upsilon^{<w}(v)\rightarrow u\!=\!v]$ that can be   proved from the basic laws of the order relation. So, for no $x$ can ${\tt B}_\Upsilon(x)$ hold.
 Now, in reality, there exists a number $\mathfrak{b}\!\in\!\mathbb{N}$ that  is   not definable  by any formula of length less than $6\ell_\Upsilon$ (since by our convention there are only finitely many formulas with length less than $6\ell_\Upsilon$). So, $\neg{\tt Def}_\Upsilon^{<\overline{6}\cdot
 \overline{\ell_\Upsilon}}(\overline{\mathfrak{b}})$ is true,  and since it is the least such number then $\forall w\!<\!\overline{\mathfrak{b}}\,{\tt Def}_\Upsilon^{<\overline{6}\cdot\overline{\ell}}(w)$ is true too. Thus, ${\tt Berry}_\Upsilon^{<\overline{6}\cdot\overline{\ell_\Upsilon}}
 (\overline{\mathfrak{b}})$ is true; and so is ${\tt B}_\Upsilon(\overline{\mathfrak{b}})$, which is a contradiction.
\end{proof}

This proof of Tarski's theorem (\ref{thm:semt}) does not use the Diagonal Lemma (and so it can be called diagonal-free in a way), though it can be debated whether the proof is genuinely circular-free or not.  By incorporating the proof of Theorem~\ref{thm:semequ} into this proof (e.g. by taking $\Upsilon\equiv\neg\Psi$), one can get a proof for the Semantic Diagonal Lemma, which is different from the standard (textbook) proofs (see~\cite{salehi}).

\subsection{\bf The Second Proof}\label{p2}
 \begin{definition}[Definable and Dominating Functions]\label{def:func}
\noindent

\noindent
A function $f\!\!:\mathbb{N}\rightarrow\mathbb{N}$ is called {\em  definable} whenever there exists a formula $\varphi(u,v)$ such that for every $m,n\!\in\!\mathbb{N}$ we have $f(m)\!=\!n \iff \mathbb{N}\vDash\varphi(\overline{m},\overline{n})$.

\noindent
A function $F\!\!:\mathbb{N}\rightarrow\mathbb{N}$ is said to {\em dominate}  a function $f\!\!:\mathbb{N}\rightarrow\mathbb{N}$, whenever there exists some $n\!\in\!\mathbb{N}$ such that $F(x)\!>\!f(x)$ holds for all $x\!\geqslant\!n$.
\hfill\ding{71}\end{definition}

Indeed, for a given countably indexed family of functions $\{f_i\!\!:\mathbb{N}\rightarrow\mathbb{N}\}_{i\in\mathbb{N}}$ one can find a function that dominates all the functions of this family:
put
\newline\centerline{$F(x)\!=\!1\!+\!\max_{i\leqslant x}f_i(x)$;} then for every $k\!\in\!\mathbb{N}$ and every $x\!\geqslant\!k$,  $f_k(x)\!\leqslant\!\max_{i\leqslant x}f_i(x)\!<\![1\!+\!\max_{i\leqslant x}f_i(x)]\!=\!F(x)$ holds. This idea is used in the following proof of Tarski's theorem (\ref{thm:semt}); cf. \cite{kotlarski98,kotlarski04}:

\bigskip

\begin{proof}

\noindent
Define the function $F\!\!:\mathbb{N}\rightarrow\mathbb{N}$ as
$$F(x)=\min\{ y \mid \forall\alpha\!\leqslant\!x[\exists z\,\alpha(x,z)
\!\rightarrow\!\exists z\!<\!y\,\alpha(x,z)]\}.$$
We show that the function $F$ dominates every definable function, but is itself definable if  (\ref{eq:special}) holds; and this is a contradiction (since no function can dominate itself). To see that $F$ dominates the family of all definable functions, assume that a function $f\!\!:\mathbb{N}\rightarrow\mathbb{N}$ is definable by a formula  $\varphi(u,v)$. Now, for every   $m\!\geqslant\!\ulcorner\varphi\urcorner$ we show that $F(m)\!>\!f(m)$ holds:  if
$F(m)\!\leqslant\!f(m)$, then from     $\varphi(\overline{m},\overline{f(m)})$
we have $\exists z \varphi(\overline{m},v)$ and so
$\exists z\!<\!\overline{F(m)}\!\!: \varphi(\overline{m},z)$ by the definition of $F$, which implies
$\exists z\!<\!\overline{f(m)}\!\!: \varphi(\overline{m},z)$
by the assumption $F(m)\!\leqslant\!f(m)$; but for every
 $k\neq f(m)$ we have $\neg\varphi(\overline{m},\overline{k})$, and so
 $\forall z\!<\!\overline{f(m)}\!\!: \neg\varphi(\overline{m},z)$, a contradiction.
Now, if (\ref{eq:special}) holds for $\Upsilon$,  then $F$   is actually definable by   
$\psi(u,v)\wedge\forall w\!<\!v\neg\psi(u,w)$ where $\psi(u,v)$ is the formula $\forall\alpha\!\leqslant\!u[\exists z\,\Upsilon(\ulcorner\alpha(u,z)\urcorner)
\rightarrow\exists z\!<\!v\,\Upsilon(\ulcorner\alpha(u,z)\urcorner)]$.
\end{proof}

As a matter of fact, the function $F$ used   by Kotlarski \cite{kotlarski98,kotlarski04} is defined as
$$F(x)=\min\{ y \mid \forall\alpha,\!u\!\leqslant\!x[\exists z\,\alpha(u,z)
\!\rightarrow\!\exists z\!<\!y\,\alpha(u,z)]\}$$
which corresponds to 
$F(x)\!=\!1\!+\!\max_{i,j\leqslant x}f_i(j)$ that dominates $\{f_i\!\!:\mathbb{N}\!\rightarrow\!\mathbb{N}\}_{i\in\mathbb{N}}$.

\section{The Diagonal Lemma, Syntactically}
\label{sec:synt}

The Diagonal Lemma is usually stated as the provability of the equivalence $\Psi(\ulcorner\theta\urcorner)\leftrightarrow\theta$ in a theory like Robinson's arithmetic, for some sentence $\theta$ which depends on the given  formula $\Psi(x)$. Let us call this the {\em Syntactic Diagonal Lemma}. A syntactic version of Tarski's theorem on the undefinability of truth is as follows:

\begin{definition}[Syntactic Tarski's Theorem]\label{def:syntt}
\noindent

\noindent
For a formula $\Psi(x)$, let ${\sf TB}^\Psi$ be the set of all truth biconditionals $\Psi(\ulcorner\beta\urcorner)\leftrightarrow\beta$,  where $\beta$ ranges over all the sentences. That is
${\sf TB}^\Psi=\{\Psi(\ulcorner\beta\urcorner)\leftrightarrow\beta \mid \beta \text{ is a sentence}\}$.

\noindent
The following statement is called the {\em Syntactic Tarski's Theorem} on a consistent  $T$:

\qquad
{\sl For every $\Psi(x)$ we have $T\nsupseteq{\sf TB}^\Psi$.}
\hfill\ding{71}\end{definition}

\begin{definition}[Weak Diagonal Lemma]\label{def:wdl}
\noindent

\noindent
The following statement is called the {\em Weak Diagonal Lemma} about a consistent theory $T$:

 \!\!
{\sl For every  $\Psi(x)$ there exists a sentence $\theta$ such that
$\Psi(\ulcorner\theta\urcorner)
\!\leftrightarrow\!\theta$ is consistent with $T$.}
\hfill\ding{71}\end{definition}

We show that   Syntactic Tarski's Theorem is equivalent to  Weak (Syntactic) Diagonal Lemma.

\begin{theorem}[Weak Diagonal Lemma $\boldsymbol\iff$ Syntactic Tarski's  Theorem]\label{thm:synequ}
\noindent

\noindent
The Weak Diagonal Lemma is equivalent to  Syntactic Tarski's Theorem.
\end{theorem}
\begin{proof}

\noindent
First, suppose that the Weak Diagonal Lemma holds for a consistent theory $T$. Take any formula $\Psi(x)$; we show that $T\nsupseteq{\sf TB}^\Psi$. There exists a sentence $\theta$ such that the theory $T$ is consistent with $\neg\Psi(\ulcorner\theta\urcorner)\leftrightarrow
\theta$. Thus, $T\nvdash\Psi(\ulcorner\theta\urcorner)\leftrightarrow
\theta$ and so $T\nvdash{\sf TB}^\Psi$.

Second, suppose that $T\nsupseteq{\sf TB}^\Phi$ for all formulas $\Phi(x)$. Take any formula $\Psi(x)$; we show the existence of some $\theta$ such that $T$ is consistent with $\Psi(\ulcorner\theta\urcorner)\leftrightarrow
\theta$. Since $T\nsupseteq{\sf TB}^{\neg\Psi}$, there should exist some sentence $\theta$ such that $T\nvdash\neg\Psi(\ulcorner\theta\urcorner)\leftrightarrow
\theta$. Therefore, $T$ is consistent with the sentence $\Psi(\ulcorner\theta\urcorner)
\!\leftrightarrow\!\theta$.
\end{proof}

As a matter of fact, the Weak Diagonal Lemma {\em cannot} show the independence of the G\"odel sentence (when the theory is sound even):

 \begin{remark}[Weak Diagonal Lemma vs. G\"odel's Proof]\label{rem:godel}{\rm
\noindent

\noindent
For a consistent and recursively enumerable theory $T$  extending Robinson's arithmetic,  the consistency of $\neg{\sf Pr}_T(\ulcorner\theta\urcorner)\leftrightarrow\theta$ with $T$  implies that $\theta$ is unprovable in $T$, but does not imply that $\theta$ is independent from $T$, even if $T$ is $\omega$-consistent:

\begin{itemize}
\item[(1)] If $T\vdash\theta$ then $T\vdash{\sf Pr}_T(\ulcorner\theta\urcorner)$, and so $T+[\neg{\sf Pr}_T(\ulcorner\theta\urcorner)\leftrightarrow\theta]
\vdash\neg\theta$, therefore  
$T+[\neg{\sf Pr}_T(\ulcorner\theta\urcorner)\leftrightarrow\theta]$ cannot be consistent.
\item[(2)]  For a contradictory sentence like  $\delta=(0\neq 0)$, the sentence $\neg{\sf Pr}_T(\ulcorner\delta\urcorner)\leftrightarrow\delta$ is consistent with $T$ (by  G\"odel's Second Incompleteness Theorem), but $\delta$ is not independent from $T$ (as $T$ proves its negation).
\end{itemize}

It is stated in \cite[p.~202]{moschovakis10}
that
\textsf{\small for every sentence $\sigma$,
 $T\vdash\sigma\iff T\vdash\neg{\sf Pr}_T(\ulcorner\sigma\urcorner)$ implies $T\nvdash\neg\sigma$ if $T$ is sound}. Unfortunately, this is not true since for e.g. $\sigma=(0\!\neq\!0)$ we do have that $T\vdash\sigma\iff T\vdash\neg{\sf Pr}_T(\ulcorner\sigma\urcorner)$ by G\"odel's second incompleteness theorem, but trivially $T\vdash\neg\sigma$ holds. If we replace ${\sf Pr}_T(x)$ with Rosser's provability predicate $R{\sf Pr}_T(x)$, then it is true that for every $\varrho$ that satisfies $T\vdash\varrho\iff T\vdash\neg R{\sf Pr}_T(\ulcorner\varrho\urcorner)$ we have $T\nvdash\varrho,\neg\varrho$ if $T$ is (only) consistent; see the next theorem.
}
\hfill\ding{71}\end{remark}

However, the Weak Diagonal Lemma is sufficiently strong to prove Rosser's theorem:

\begin{theorem}[Weak Diagonal Lemma $\boldsymbol\Longrightarrow$ Rosser's  Theorem]\label{thm:wdlross}
\noindent

\noindent
If the Weak Diagonal Lemma holds for a
 consistent and recursively enumerable theory that extends Robinson's arithmetic, then there exists a sentences which is independent from that theory.
\end{theorem}
\begin{proof}

\noindent
For such a theory $T$, suppose that ${\sf prf}_T(x,y)$ is its proof predicate (stating that $x$ is the G\"odel code of a proof of the sentence with G\"odel code $y$ in $T$). By the Weak Diagonal Lemma there exists a sentence $\rho$ such that the following theory is consistent: $$U=T+\Big(\forall x\big[{\sf prf}_{T}(x,\ulcorner\rho\urcorner)\rightarrow\exists y\!<\!x\, {\sf prf}_{T}(y,\ulcorner\neg\rho\urcorner)\big]
\longleftrightarrow\rho\Big).$$
The standard proof of Rosser's theorem can show that $\rho$ is independent from $T$:
\begin{itemize}
\item If $T\vdash\rho$, then  $T\vdash{\sf prf}_{T}(\overline{k},\ulcorner\rho\urcorner)$ for some $k\!\in\!\mathbb{N}$ and so
    $U\vdash\exists y\!<\!\overline{k}\, {\sf prf}_{T}(y,\ulcorner\neg\rho\urcorner)$,
    by the definition of $U$,  which contradicts $\bigwedge\hspace{-0.8em}\bigwedge_{m\in\mathbb{N}}
    \; U\vdash\neg{\sf prf}_{T}(\overline{m},\ulcorner\neg\rho\urcorner)$  (that holds by $T\nvdash\neg\rho$).
\item If $T\vdash\neg\rho$, then $T\vdash{\sf prf}_{T}(\overline{k},\ulcorner\neg\rho\urcorner)$ for some $k\!\in\!\mathbb{N}$.
    Reason inside  $U$:
\begin{itemize}
\item[]
    for some $x$ we   have
     ${\sf prf}_{T}(x,\ulcorner\rho\urcorner)\wedge\forall y\!<\!x\,\neg{\sf prf}_{T}(y,\ulcorner\neg\rho\urcorner)$; now  $\overline{k}\!<\!x$ is impossible,
    and so   $x\!\leqslant\!\overline{k}$,
    whence $\bigvee\hspace{-0.8em}\bigvee_{i\leqslant k}\,(x=\overline{i})$, therefore  $\bigvee\hspace{-0.8em}\bigvee_{i\leqslant k}\,{\sf prf}_{T}(\overline{i},\ulcorner\rho\urcorner)$.
\end{itemize}
Hence,   $U\vdash\bigvee\hspace{-0.8em}\bigvee_{i\leqslant k}\,{\sf prf}_{T}(\overline{i},\ulcorner\rho\urcorner)$, but this contradicts $\bigwedge\hspace{-0.8em}\bigwedge_{m\in\mathbb{N}} \,  U\vdash\neg{\sf prf}_{T}(\overline{m},\ulcorner\rho\urcorner)$
    (that holds by $T\nvdash\rho$).
\end{itemize}
Therefore, $T\nvdash\rho,\neg\rho$.
\end{proof}

So, the Weak Diagonal Lemma is worthy of  studying further.
Unfortunately, the second proof (Subsection~\ref{p2}) for Tarski's Theorem cannot be carried over to the Syntactic Tarski's Theorem (cf.~\cite{visser19}). However, the first proof (Subsection~\ref{p1}) can be adapted for it:

\begin{theorem}[Syntactic Tarski's Theorem]\label{thm:p1stt}
\noindent

\noindent
If $T$ is a consistent extension of Robinson's arithmetic, then for no formula $\Psi(x)$ can we have  $T\supseteq{\sf TB}^\Psi$.
\end{theorem}
\begin{proof}

\noindent
 Assume that the consistent theory $T$ contains Robinson's arithmetic, and it also contains the set ${\sf TB}^\Upsilon$  for a formula $\Upsilon(x)$. We work with the Convention of Subsection~\ref{p1} (and also Definition~\ref{def:berry}). Let $\mathfrak{q}$ be the term $\overline{6}\cdot\overline{\ell_\Upsilon}$ which represents  the number $6\ell_\Upsilon$. Fix a number $n\!\in\!\mathbb{N}$; and reason inside the theory $T$:
 \begin{itemize}
 \item[]
 Assume ${\tt B}_\Upsilon(\overline{n})$; so,    ${\tt Berry}_\Upsilon^{<\mathfrak{q}}(\overline{n})$ thus
 (1)~$\neg{\tt Def}_\Upsilon^{<\mathfrak{q}}(\overline{n})$ and (2)~$\forall w\!<\!\overline{n}\,{\tt Def}_\Upsilon^{<\mathfrak{q}}(w)$ hold.
 Fix $\zeta$; if ${\tt B}_\Upsilon(\zeta)$ holds, then
(i)~$\neg{\tt Def}_\Upsilon^{<\mathfrak{q}}(\zeta)$ and (ii)~$\forall w\!<\!\zeta\,{\tt Def}_\Upsilon^{<\mathfrak{q}}(w)$.
 We also have
 either $\zeta\!<\!\overline{n}$ or $\zeta\!=\!\overline{n}$ or $\zeta\!>\!\overline{n}$ (which holds  in Robinson's arithmetic). Now, $\zeta\!<\!\overline{n}$ contradicts (i) and (2), and $\zeta\!>\!\overline{n}$ contradicts (1) and (ii). Whence, $\zeta\!=\!\overline{n}$; which shows that  $\forall\zeta[{\tt B}_\Upsilon(\zeta)\leftrightarrow\zeta\!=\!\overline{n}]$ holds.  Thus,  ${\tt D}(\ulcorner{\tt B}_\Upsilon\urcorner,\overline{n})$ holds and since ${\tt len}({\tt B}_\Upsilon(x))\!<\!\mathfrak{q}$, then we have ${\tt Def}_\Upsilon^{<\mathfrak{q}}(\overline{n})$ by ${\sf TB}^\Upsilon$; which contradicts (1). Therefore, the assumption ${\tt B}_\Upsilon(\overline{n})$ leads to a contradiction.
 Whence, $\neg{\tt B}_\Upsilon(\overline{n})$ holds, so $\neg{\tt Berry}_\Upsilon^{<\mathfrak{q}}(\overline{n})$, thus   we have ($\ast$) $\bigwedge\hspace{-0.8em}\bigwedge_{i<n}{\tt Def}_\Upsilon^{<\mathfrak{q}}(\overline{i})
 \rightarrow
 {\tt Def}_\Upsilon^{<\mathfrak{q}}(\overline{n})$. \end{itemize}
Therefore,  $T\vdash{\tt Def}_\Upsilon^{<\mathfrak{q}}(\overline{n})$  can be shown by induction on $n\!\in\!\mathbb{N}$ from $(\ast)$. Let $\mathfrak{p}\!\in\!\mathbb{N}$ be  greater than all the G\"odel codes of formulas with length less than $6\ell_\Upsilon$. Therefore, for all $n\!\in\!\mathbb{N}$ we have $T\vdash\exists\alpha\!<\!\overline{\mathfrak{p}}\,
\Upsilon(\ulcorner{\tt D}(\alpha,\overline{n})\urcorner)$. So, inside $T$ for any $n\!\in\!\mathbb{N}$ there exists some formula $\alpha_n(x)$ that defines $n$, i.e., $\forall\zeta[\alpha_n(\zeta)
\leftrightarrow\zeta\!=\!\overline{n}]$
 holds by ${\sf TB}^\Upsilon$, and the G\"odel codes of all $\alpha_n$'s are less than $\mathfrak{p}$. This contradicts the Pigeonhole's Principle a version of which is provable in Robinson's arithmetic: since both of the sentences $\forall x (x\nless 0)$ and $\forall x(x\!<\!\overline{k\!+\!1}\rightarrow x\!\leqslant\!\overline{k})$ are provable in this arithmetic, then for every $\{\alpha_k\!<\!\overline{\mathfrak{p}}\}_{k\leqslant\mathfrak{p}}$
 there should exist some $i\!<\!j\!\leqslant\!\mathfrak{p}$ such that $\alpha_i\!=\!\alpha_j$. Reason inside $T$ again:
  \begin{itemize}
 \item[] There are some $\overline{i}\!<\!\overline{j}\!\leqslant\!\overline{\mathfrak{p}}$ and some formula $\varphi(x)$ for which we have $\Upsilon(\ulcorner{\tt D}(\ulcorner\varphi\urcorner,\overline{i})\urcorner)$ and $\Upsilon(\ulcorner{\tt D}(\ulcorner\varphi\urcorner,\overline{j})\urcorner)$. So, by ${\sf TB}^\Upsilon$, both $\forall\zeta[\varphi(\zeta)
     \leftrightarrow\zeta\!=\!\overline{i}]$ and
     $\forall\zeta[\varphi(\zeta)
     \leftrightarrow\zeta\!=\!\overline{j}]$ hold, combining which implies that $\overline{i}\!=\!\overline{i}
     \rightarrow\varphi(\overline{i})
     \rightarrow\overline{i}\!=\!\overline{j}$, a contradiction.
 \end{itemize}
So, $T$ is inconsistent, which is a contradiction with the assumption.
\end{proof}

\section{Conclusion}
\label{sec:conc}
The semantic form of Tarski's Undefinability Theorem,
{\sl that the set $\{\ulcorner\eta\urcorner\mid
\mathbb{N}\vDash\eta\}$  is not definable in arithmetic}, is equivalent to the semantic form of the Diagonal Lemma,
that {\sl for a given $\Psi(x)$ there exists a sentence $\theta$ such that  $\mathbb{N}\vDash\Psi(\ulcorner\theta\urcorner)
\leftrightarrow\theta$}. We outlined two  seemingly diagonal-free proofs for these equivalent theorems. The syntactic form of Tarski's Theorem, that {\sl no consistent extension of Robinson's arithmetic contains the set of truth biconditionals ${\sf TB}^\Psi=\{\Psi(\ulcorner\beta\urcorner)
\leftrightarrow\beta \mid \beta \text{ is a sentence}\}$}, is equivalent to the Weak (syntactic) Diagonal Lemma, that {\sl for every  $\Psi(x)$ there exists a sentence $\theta$ such that
$\Psi(\ulcorner\theta\urcorner)
\leftrightarrow\theta$ is consistent with such a theory}. Even though  G\"odel's proof does not work with the Weak Diagonal Lemma, the weak lemma  is sufficiently strong to prove Rosser's theorem. So, the syntactic form of Tarski's theorem can derive G\"odel and Rosser's incompleteness theorem (by combining the proofs of Theorems~\ref{thm:p1stt}, \ref{thm:synequ} and \ref{thm:wdlross}).

\paragraph{\textsf{\textsc{Acknowledgements:}}}
This research is supported by the Office of the Vice Chancellor for Research and Technology, University of Tabriz, IRAN.

\end{document}